\documentclass[11pt]{amsart}
\usepackage[utf8]{inputenc}

\usepackage{mathrsfs}
\usepackage{parskip}
\usepackage[margin=1in]{geometry}

\usepackage{enumerate}
\usepackage{amsthm}
\usepackage{amsmath}
\usepackage{amsfonts}
\usepackage{amssymb}
\usepackage{graphicx}
\usepackage{color}
\usepackage{graphics}
\usepackage{eepic}
\usepackage{nicefrac}

\usepackage{tikzsymbols}

\newcommand{\ignore}[1]{}

\usepackage[colorinlistoftodos]{todonotes}

\newcommand{\commE}{\todo[inline, color=green!40]}

\renewcommand{\Re}{\operatorname{Re}}

\newcommand{\e}{\varepsilon}

\newcommand{\C}{{\mathbb{C}}}
\newcommand{\R}{{\mathbb{R}}}

\newcommand{\D}{{\mathbb{D}}}

\newcommand{\p}{\partial}
\newcommand{\ol}{\overline}

\usepackage{mathtools} 
\mathtoolsset{showonlyrefs,showmanualtags}

\newtheorem{thm}{Theorem}[section]

\newtheorem{conj}[thm]{Conjecture}
\newtheorem{prop}[thm]{Proposition}

\newtheorem{lemma}[thm]{Lemma}

\newtheorem{prob}[thm]{Problem}

\theoremstyle{definition}
\newtheorem{defn}[thm]{Definition}

\theoremstyle{remark}
\newtheorem{remark}[thm]{Remark}

\title[Critical points]{A note on the critical points of the localization landscape}

\author{Erik Lundberg}
\address{Department of Mathematical Sciences, Florida Atlantic University, Boca Raton, USA-33431.}
\email{elundber@fau.edu}

\author{Koushik Ramachandran}
\address{Tata Institute of Fundamental Research, Centre for Applicable Mathematics, Bengaluru, India-560065}
\email{koushik@tifrbng.res.in}

\begin{document}
\maketitle

\begin{abstract}
Let $\Omega\subset\C$ be a bounded domain.  In this note, we use complex variable methods to study the number of critical points of the function $v=v_\Omega$
that solves the elliptic problem $\Delta v = -2$ in $\Omega,$ with boundary values $v=0$ on $\partial\Omega.$
This problem has a classical flavor but is especially motivated by recent studies on localization of eigenfunctions.
We provide an upper bound on the number of critical points of $v$ when $\Omega$ belongs to a special class of domains in the plane, namely, domains for which the boundary $\partial\Omega$ is contained in $\{z:|z|^2 = f(z) + \overline{f(z)}\},$ where $f'(z)$ is a rational function.  We furnish examples of domains where this bound is attained. We also prove a bound on the number of critical points in the case when $\Omega$ is a quadrature domain.
The note concludes with the statement of some open problems and conjectures.
\end{abstract}

\section{Introduction}
Let $\Omega\subset\mathbb{C}$ be a finitely-connected bounded domain,
and consider the solution $v$ to the following boundary value problem
\begin{equation} 
\label{eq:landscape}
\begin{cases}
\Delta v(z) = -2, & z\in\Omega \\
v(z) =0, & z\in\partial\Omega.
\end{cases}
\end{equation}
We will refer to $v$ as
the \emph{localization landscape function} or simply \emph{landscape function}
of $\Omega$ (this choice of terminology is explained in the next paragraph).
It is a simple fact that such a solution $v$ is unique. In this paper, we study the number of critical points of $v$ in $\Omega,$ i.e., the number of solutions to $\nabla v(z)=0$ for $z\in\Omega$. Observe that $v$ is a smooth superharmonic function in $\Omega$ that vanishes on the boundary.
By the minimum principle, $v> 0$ in $\Omega$ and therefore attains its maximum on $\overline\Omega$ at a point $z_0\in\Omega.$ This guarantees that there is at least one solution to $\nabla v(z)=0$. 
By superharmonicity we also have that $v$ does not have any minima
in $\Omega$ (only maxima and saddles).
Note further that by the Hopf lemma, 
$\nabla v$ cannot vanish on the boundary $\partial\Omega.$
Since $\partial \Omega$ is a level set of $v$
this implies that $\nabla v$ is parallel to the inward pointing normal so that $\nabla v$ makes a single revolution along each component of $\p \Omega$.
This allows an application of index theory in order to relate the number of maxima to the number of saddles (given the connectivity of $\partial \Omega$), see Theorem \ref{thm:Morse} below.
In particular, this produces a lower bound on the number of critical points of $v$ in terms of the connectivity of $\Omega$.

The problem of counting critical points of solutions to elliptic PDE is classical, see the survey \cite{Magnanini} and the references therein,
but our primary motivation for studying the critical points of the solution to \eqref{eq:landscape}
comes from recent studies on localization of eigenfunctions \cite{Mayboroda}, \cite{ArnoldDavidFilJerMay}, 
\cite{ArnoldDavidJerMayFil},
\cite{MayFilclamped}, \cite{Steinerberger}.
Namely, solutions to  \eqref{eq:landscape} have been proposed as a ``universal landscape'' governing the localization phenomena observed for eigenfunctions of the Laplacian.
The localization landscape (graph of $v$) is used to construct a valley network (a cell complex with vertices given by saddles and edges constructed from steepest descent curves along with arcs from $\partial \Omega$),
and each cell serves as a candidate region for localized high-amplitude activity exhibited by one of the eigenfunctions (see \cite{Mayboroda} for details).
The complexity of the valley network is closely tied to the critical points of $v$ (one may even use the maxima themselves as rudimentary 
predictors for locating high-amplitude regions).
Thus, it would be useful to have results estimating the number of critical points of $v$ in terms of basic geometric features of the domain $\Omega$.

When $\Omega$ is simply-connected the landscape function $v$ coincides with the so-called \emph{torsion function} of $\Omega$,
used in elasticity theory to model the shear stress
in a cylindrical bar with uniform cross section $\Omega$
undergoing a twisting force \cite{Musk}.
In this context, the critical points of $v$ correspond to the physical locations where no stress is felt.
Note that the torsion function of elasticity theory is defined differently in the multiply-connected case (having distinct Dirichlet conditions on separate boundary components); since
we are primarily motivated by the above-mentioned application to localization, we take vanishing Dirichlet condition on all boundary components in \eqref{eq:landscape}.

When $\Omega$ is convex, 
$v$ has a unique critical point.
It appears that very little is known beyond this.
It would be desirable to 
bound the number of critical points of $v$ in terms of
geometric features of $\partial \Omega$.
For instance, for $\Omega$ simply-connected
and $\partial \Omega$ smooth, is there an upper bound on the
number of critical points of $v$ in terms of the number of oscillations of $\p \Omega$ (defined as oscillations of the distance between $\p \Omega$ and a fixed reference point)?

Alternatively,
one may restrict attention to an interesting class of domains admitting an algebraic description
(accompanied by a notion of degree or order) and attempt to bound the number of critical points of $v$
in terms of the order of $\partial \Omega$.
Some natural classes that come to mind are quadrature domains, lemniscate domains, or the examples discussed in \cite{FleemanKhavinson}
for which $\nabla u$ turns out to be 
the complex conjugate of a rational function in $z$, where $u$ is the harmonic part of $v$ defined below in \eqref{eq:harm}.
In this note, 
we focus primarily on the latter class of domains, which we refer to as \emph{RL-domains}
(``rational landscape domains'').

If $v$ is the solution to \eqref{eq:landscape}, 
it is easy to see then that $u(z) = v(z) + \frac{|z|^2}{2}$ solves
\begin{equation}\label{eq:harm}
\begin{cases}
\Delta u(z) = 0, & z\in\Omega \\
u(z)= \frac{|z|^2}{2}, & z\in\partial\Omega.
\end{cases}
\end{equation}
Therefore, critical points of $v$ are the same as solutions to 
\[\label{eq:fixed}
\nabla u(z) = z, \quad z\in\Omega.
\]
This is an anti-holomorphic fixed point problem,
placing it within a class of problems that has seen many breakthroughs over the past fifteen years utilizing anti-holomorphic dynamics
(the primary technique in these studies is an adaptation of the Fatou theorem from holomorphic dynamics).

In \cite{KhSw},
the authors used anti-holomorphic dynamics to address a question of Sheil-Small \cite{Sh-Sm}
on the number of zeros of harmonic polynomials.
An extension of this result to the setting of harmonic rational functions in \cite{KhavinsonNeumann}
settled a conjecture \cite{Rhie} in astronomy (concerning gravitational lensing by point-mass configurations).
Transcendental dynamics were used in 
\cite{KhLu2010} to study 
another problem in gravitational lensing
involving elliptical galaxies
(a sharp result was obtained in \cite{BergErem2010}).
Utilizing quasiconformal surgery, the use of
anti-holomorphic dynamics as an indirect technique was further adapted
to study the topology of
quadrature domains in \cite{LeeMak}.
Anti-holomorphic dynamics involving elliptic functions were used in \cite{BergErem2016}
to give much shorter and greatly simplified proofs of the results in \cite{LW1}, \cite{LW2}, \cite{LW3} on the number of critical points of Green's function on a torus.
Transcendental dynamics were used again in
\cite{BergErem2018} to completely classify the number of zeros of certain complex equations involving polynomials and logarithms.
Prior to these applications,
the study of anti-holomorphic dynamics had been initiated in \cite{Crowe}
(see also \cite{Sch1}, \cite{Sch2}).

Among the studies mentioned above, the one \cite{BergErem2016} concerning
the critical points of Green's function on a torus is especially relevant.
The ``Green's function'' in that context is not harmonic but solves a Poisson equation with constant source term (plus a delta-mass).
Taking a sublevel set of (the negation of) this Green's function (in order to remove a neighborhood of the logarithmic pole)
and adding a constant so that the resulting function vanishes on the boundary of the sublevel set,
we obtain the landscape function of a domain on a torus.
The results of \cite{BergErem2016}
give a complete description of the possible number of critical points.

\subsection{Outline of the paper}
We review some preliminary results in Section \ref{sec:prelim}.
In Section \ref{sec:RL}
we prove an upper bound on the number of critical points of $v$
for so-called RL-domains,
and in Section \ref{sec:quad}
we prove an upper bound
for certain quadrature domains.
We consider some additional examples  in Section \ref{sec:examples},
and we conclude with some open problems
in Section \ref{sec:concl}.

\subsection*{Acknowledgements}
We thank Alexandre Eremenko, Dmitry Khavinson, and Svitlana Mayboroda for helpful comments.

\section{Preliminaries}\label{sec:prelim}

For the next two results
we refer to \cite[Thm. 3.3 and Cor. 3.4]{AlMa}.

\begin{thm}\label{thm:Morse}
Let $\Omega$ be a 
bounded domain in the plane 
whose boundary $\p \Omega$ consists of
$k$ connected components each of class $C^{1,\alpha}$.
If the critical points of the landscape function $v$
are isolated, then
$S - M = k-2$,
where $S$ and $M$ denote
the number of saddles and maxima
(respectively) of $v$.
\end{thm}

As explained in the introduction, the above result can be viewed as a consequence of superharmonicity, the Hopf Lemma, and the Poincare-Hopf index theorem\footnote{Among complex analysts, the index theorem is often referred to as the ``generalized argument principle'' \cite[Sec. 2.5]{Sh-Sm}.  Alternatively, one can use Morse theory instead of index theory.}.

\begin{thm}\label{thm:finite}
Let $\Omega$ be a 
bounded simply-connected domain in the plane with boundary $\p \Omega$ of class $C^{1,\alpha}$.
Then the number $N$ of critical points of the landscape function $v$ is finite.
\end{thm}


If the assumption that
$\Omega$ is simply-connected is removed then the critical set
may be infinite.
For instance, in the case of an annulus the critical set is 
a circle.
More generally, 
prescribing an arbitrary
non-singular
analytic Jordan curve $\gamma$,
one may construct as follows an annular neighborhood $\Omega$ of $\gamma$
whose landscape function
$v$ has $\gamma$ as its critical set.
Let $S(z)$ denote the Schwarz function of $\gamma$.
Then $\ol{S(z)}$ is an anti-conformal reflection function that fixes $\gamma$ \cite{Davis}.

The function $\ol{S(z)} - z$ is the gradient
$\nabla w$ of the (unique) solution
$w$ to the Cauchy problem
\begin{equation}
\begin{cases}
\Delta w(z) = -2, & z \text{ near } \gamma \\
w(z) = 0, \nabla w(z) = 0, & z\in \gamma.
\end{cases}
\end{equation}
For $z$ near $z_0 \in \gamma$,
we have that $\ol{S(z)}$
is, up to first order,
the same as standard reflection over the tangent line to $\gamma$ at $z_0$
(see \cite[p. 47]{Davis}).
Thus, we have $\ol{S(z)} - z = -2 |z-z_0| \hat{n} + O(|z-z_0|^2),$
where $\hat{n}$ denotes the unit normal vector to $\gamma$ at $z_0$.
This implies (by expressing $w(z)$ as a path integration
of $\nabla w$)
that $w(z)<0$ for all $z$ sufficiently near
(but not on) $\gamma$.
Choose $\Omega = \{ w(z) > -\e \}$
with $\e>0$ sufficiently small.
Then $v(z) = w(z) + \e$
solves \eqref{eq:landscape}
and carries $\gamma$ as its critical set.

\begin{remark}\label{rmk:generic}
The above pathological examples might seem plentiful since we were free to prescribe the analytic curve $\gamma$ of critical points.
However, we point out that these examples are unstable under perturbation.
Namely, following a standard argument from Morse theory \cite[Lemma 2.21]{Matsumoto} we can obtain a Morse function (Morse functions have isolated nondegenerate critical points) by adding an arbitrarily small \emph{linear} perturbation to $v$.
Since the perturbation is linear, the perturbed $v$ still solves the same PDE, and with an appropriate corresponding  perturbation of $\Omega$ the vanishing boundary condition is also satisfied.
Using further tools from Morse theory \cite[Lemma 2.26]{Matsumoto}, one can also show that the ``good'' domains whose landscape function $v$ is a Morse function (and hence has isolated nondegenerate critical points) are stable under perturbation.
We do not pursue these genericity results in detail as they are not needed for the primary class of domains we consider (where we can more directly rule out pathological examples, see the proof of Theorem \ref{thm:RL}).
\end{remark}

\section{The number of critical points of $v$: 
an upper bound for RL-domains}\label{sec:RL}


Let us consider the special class of domains that were introduced in \cite{FleemanKhavinson}, for which the solution to the harmonic Dirichlet problem \eqref{eq:harm} with data $|z|^2$ is the real part of a function whose derivative is rational in $z$.

\begin{defn}\label{RL}
We say that a domain $\Omega$ is an \emph{RL-domain} (``rational landscape domain'') if the boundary $\partial \Omega$ of $\Omega$ is a subset of 
$$\{z: |z|^2 = f(z) + \overline{f(z)}\},$$
where the derivative $f'(z)$ of $f(z)$ is a rational function.
Given a domain $\Omega$ of this type, we define the \emph{order} of $\Omega$ to be the degree of the derivative $f'(z)$ of the associated function $f$.
\end{defn} 

We note that even though these domains have a special form, they are dense in the space of Jordan domains with smooth boundary, see Proposition \ref{p1}
below. Also, note that for an RL-domain, the landscape function is given by $v(z) = \Re f(z) - \frac{|z|^2}{2}.$ The following is directly related to a result of Khavinson and Neumann \cite{KhavinsonNeumann} on the number of solutions of an equation of the form $r(z)+\bar{z}=0,$ where $r$ is a rational function of degree $n\geq 2.$

\begin{thm}\label{thm:RL}
Suppose $\Omega = \{z \in \C: f(z) + \ol{f(z)} -|z|^2 > 0 \}$ is a smooth RL-domain of order $n\geq 2,$ and connectivity $k$.
If $\Omega$ is not an annulus,
then the number $N$ of critical points of $v$ satisfies
$$ N \leq 4n + k - 6.$$
\end{thm}

\begin{proof}
The assumption that $\Omega$
is not an annulus
guarantees that the critical set is finite.
Indeed, suppose the critical set is infinite.
Then $\nabla v = \ol{f'(z)} - z$ must vanish on an analytic arc $\gamma$ (see the proof of Corollary 3.4 in \cite{AlMa}).  We notice that the Schwarz function of $\gamma$ is $f'(z)$.
Since $f'(z)$ is rational, this implies that $\gamma$ is a circle \cite[Ch. 10]{Davis} and, after a complex linear change of variables, we have $f'(z) = 1/z$.  Hence $u(z) = \log |z| + C$, for $C$ a constant, so that the set $|z|^2/2 = u(z)$ defines an annulus.

The locations where maxima of $v$ occur correspond to non-repelling fixed points of the function $\ol{F(z)}$, where $F(z) = f'(z)$.
Indeed, if $z_0$ is a critical point
then we have that $\nabla v = 0$,
or equivalently $\ol{f'(z_0)}=z_0$.
i.e., $z_0$ is a fixed point of $\ol{F(z)}$.
Moreover, 
if $v$ attains a maximum at $z_0$
then 
the determinant of its Hessian matrix is non-negative
at $z_0$.
The Hessian matrix of $v$
is the Jacobian matrix of the harmonic mapping
$\ol{F(z)} - z$,
and the Jacobian determinant 
of this harmonic mapping is $1-|F'(z)|^2$.
Thus, if a maximum occurs at $z_0$, then
$1 - |F'(z_0)|^2 \geq 0$,
or equivalently, $|F'(z_0)| \leq 1$.
i.e., $z_0$ is a non-repelling fixed point of $\ol{F(z)}$.

This along with \cite[Prop. 1]{KhavinsonNeumann}
implies that the number $M$ of maxima of $v$
is at most $2n-2$.
Applying Theorem \ref{thm:Morse},
we then have 
\begin{align*}
    S &= M + k - 2 \\
    &\leq (2n-2) + k-2 \\
    &= 2n+k-4,
\end{align*}
where $S$ denotes the number of
saddles of $v$ in $\Omega$.
Thus, the total number $N$ of critical
points satisfies
$$N = M+S \leq 2n-2 + 2n+k-4 = 4n+k-6,$$
as desired.
\end{proof}

\begin{remark}
Requiring that the boundary is smooth in Theorem \ref{RL} is not too restrictive of a condition. The boundary $\partial\Omega$ of an RL-domain is given by $\{z: \Re(f) = \frac{|z|^2}{2}\},$ i.e., a level curve of a real analytic function $\psi$. If $\p \Omega$ is non-singular, i.e., $\nabla \psi$ does not vanish on $\p \Omega$, then  $\partial\Omega$ is smooth.
\end{remark}

\noindent {\bf Extremal examples:}
We exhibit an RL-domain of order $n \geq 2$ and connectivity $k=n+1$ for which the above bound is attained, so that the number of critical points is exactly $4n+k-6 = 5n-5.$ Our construction is an adaptation of the
extremal examples from gravitational lensing
that were constructed by S. Rhie \cite{Rhie}. For simplicity we assume $n \geq 3$ (the case $n=3$ can be dealt with separately).  Let 
\begin{equation}\label{eq:w}
w(z) = \epsilon\log|z| + \sum_{j=1}^{n-1}\log|z-a_j|- \frac{|z|^2}{2},
\end{equation} 
where $a_j = a e^{2\pi i j / (n-1)}$ for $j=1,...,n-1$, with $a > 0$ small (and depending on $n,$ see \cite{Bleher}), and with $a >> \epsilon > 0$.  For $T \gg 1$, consider the set 
$\Omega_T = \{z: w(z) + T > 0\}.$ All the conditions in Definition \ref{RL} hold, so that $\Omega_T$ is a bounded RL-domain of order $n$ and connectivity $n+1$. A simple verification yields that the landscape function of $\Omega_T$ is $v(z) = w(z) + T$, and
$f'(z) = \frac{\epsilon}{z} + \sum_{j=1}^{n-1}\frac{1}{z-a_j}$ is the associated rational function. As we observed before, the equation $\nabla v(z) =0$ is the same as $f'(z) = \overline{z}.$ 
By the known result \cite{Rhie}, this latter equation has $5n - 5$ solutions in the plane. A priori these critical points may not all be contained in $\Omega_T$, but we can ensure that this is the case by making $T$ larger. Therefore for this particular choice, the number of critical points of $v$ is exactly $5n-5$ which concludes the proof. 

It is more difficult to verify the existence of examples (conjectured below) that attain the bound for each $n$ and with every connectivity $k=1,2,...,n+1$,
but based on computer simulations a promising strategy is as follows: first perturb the coefficients of the logarithmic terms in \eqref{eq:w} in order to break the symmetry (and ensure the saddles have distinct heights), and then obtain domains of each connectivity by 
gradually decreasing the level $T$
so that the connectivity of $\Omega$ experiences $n$ bifurcations at $n$ different heights corresponding to saddle points.  Such bifurcations could, a priori, disconnect $\Omega$,
and the difficulty is to show that $\Omega$ in fact remains connected as $T$ crosses the heights of the $n$ highest saddles.

\begin{conj}
For every $n\geq 2,$ and for each $k=1,2,...,n+1$ there exists an RL-domain $\Omega$ of order $n$ and connectivity $k$ whose landscape function has exactly $4n + k -6$ critical points in $\Omega$.
\end{conj}

Note that $k$ cannot exceed $n+1$
since each bounded component of the complement of an RL-domain contains
a singularity of $f'$ 
(see \cite[Thm. 3]{FleemanKhavinson}).

We show some examples of extremal domains and their landscape functions in Figures \ref{fig:order3} and \ref{fig:4connected}.
The domain on the left in Figure \ref{fig:order3}
is given by
$$\Omega_L =  \left\{ z \in \C : |z|^2/2 - \frac{3}{4} \sum_{k=0}^{2} \log|z - e^{2 k \pi i/3}| < 1/2 \right\},$$
and the landscape function is
$v_L(z) = 1/2 - |z|^2/2 + \frac{3}{4} \sum_{k=0}^{2} \log|z - e^{2k \pi i/3}|$.

The domain on the right in Figure \ref{fig:order3}
is given by
$$\Omega_R =  \left\{ z \in \C : |z|^2/2 - \frac{3}{4} \sum_{k=0}^{2} \log|z - e^{2k \pi i/3}| < 2/5 \right\},$$
and the landscape function is
$v_R(z) = 2/5 - |z|^2/2 + \frac{3}{4} \sum_{k=0}^{2} \log|z - e^{2k \pi i/3}|$.

The domain shown in Figure \ref{fig:4connected}
is given by
$$\Omega = \left\{ z \in \C : |z|^2/2 - \frac{1}{3} \sum_{k=0}^{4} \log|z - e^{2k \pi i/5}| - \frac{9}{40} \log|z| = 1/2 \right\},$$
and the landscape function is
$$v(z) = 1/2 - |z|^2/2 + \frac{1}{3} \sum_{k=0}^{4} \log|z - e^{2k \pi i/5}| + \frac{9}{40} \log|z|.$$

\begin{figure}
    \centering
    \includegraphics[scale=0.4]{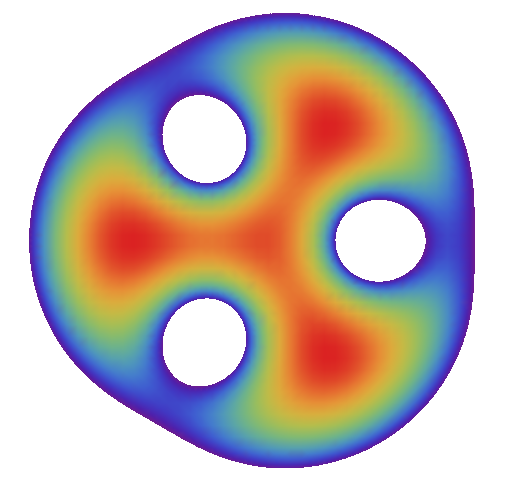}
        \includegraphics[scale=0.29]{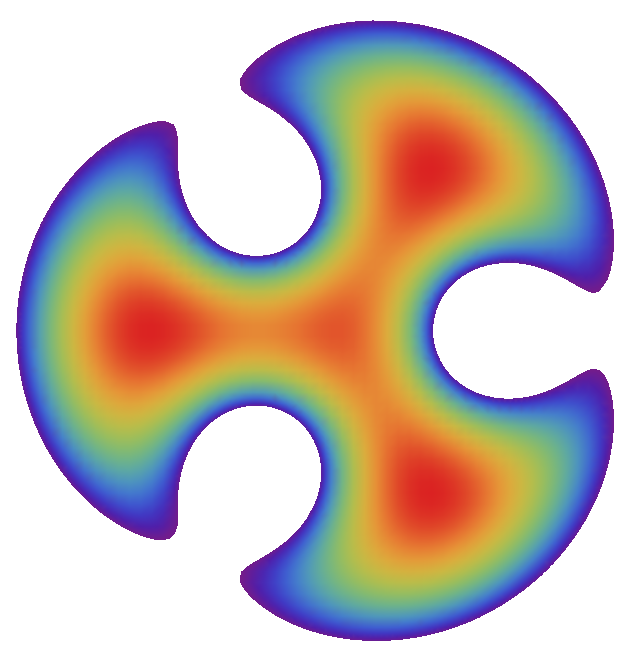}
    \caption{Left: An  RL-domain of order $n=3$ and connectivity $k=4$ whose torsion function has $4$ maxima and $6$ saddles. Right: An RL-domain of order $n=3$ and connectivity $k=1$ whose torsion function has $4$ maxima and $3$ saddles.}
    \label{fig:order3}
\end{figure}

\begin{figure}
    \centering
        \includegraphics[scale=0.55]{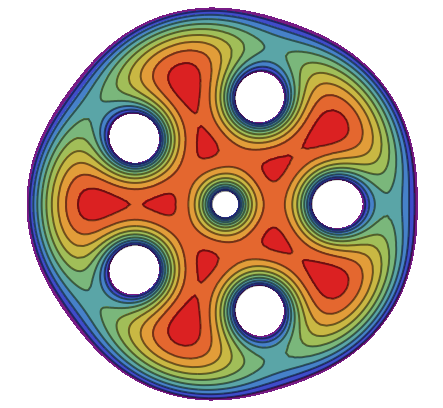}
    \caption{An RL-domain of order $n=6$ and connectivity $k=7$ along with contour plot of the landscape function $v$.  
    The landscape function has $10$ maxima and $15$ saddles.}
    \label{fig:4connected}
\end{figure}

\begin{prop}\label{p1}
Let $\Omega\subset\mathbb{C}$ be a bounded simply connected domain whose  boundary $\partial{\Omega}$ is analytic. Then, for any $\e > 0$ there is an RL-domain $D$ such that

$$d_H(D, \Omega) < \e,$$ where $d_H(A, B)$ denotes the Hausdorff distance between the sets $A$ and $B.$
\end{prop}

\begin{proof}
Let $w$ be the landscape function of $\Omega$. Then $u(z) = w(z) + \frac{|z|^2}{2}$ solves
\begin{equation}
\label{eq:harm2} \begin{cases}
\Delta u(z) = 0, & z\in\Omega \\
u(z)= \frac{|z|^2}{2}, & z\in\partial\Omega.
\end{cases}
\end{equation}
Since $\partial{\Omega}$ is analytic, $u$ has a harmonic extension (using the Schwarz function) to a bounded simply connected domain $\widetilde\Omega$ which contains $\Omega$ as a proper subdomain. Let $g = u + iv$ be an analytic completion of $u$ in $\widetilde\Omega$. Let $2\mu = dist(\bar{\Omega},\partial\widetilde\Omega)$ and $\Omega_1 = \{z\in\widetilde\Omega: dist(z, \bar{\Omega}) < \mu\}$. Let $\delta>0$ be a constant to be chosen later. By Runge's theorem, we can find a rational function $R(z)$ whose poles lie outside $\Omega_1$ with $||g-R||_{\Omega_1} < \delta$. By Cauchy's estimate, $||g'-R'||_{\bar{\Omega}} < \frac{\delta}{\mu}$. In particular,

$$||u-\Re R||_{\bar{\Omega}}< \delta,\hspace{0.1in} ||\nabla{u}-\nabla{\Re R}||_{\bar{\Omega}} < \frac{\delta}{\mu}.$$

Setting $\tau(z) = \Re R(z) - \frac{1}{2}|z|^2$, we obtain from the above displayed equation that
\begin{equation}\label{eq1}
||w-\tau||_{\bar{\Omega}}< \delta,\hspace{0.1in} ||\nabla{w}-\nabla{\tau}||_{\bar{\Omega}} < \frac{\delta}{\mu}  .
\end{equation}

Now let $\e >0$ be given. We want to show that there exists an RL-domain which approximates $\Omega$ within Hausdorff distance $\e.$ Recall that $\{w=0\} = \partial\Omega$ and $\{w > 0\} = \Omega$. From \eqref{eq1}, we get \begin{equation}\label{eq2}
\Omega_{\delta}=\{-\delta < w < \delta\}
\end{equation}
contains a component of $\{\tau = 0\}$. The
Hopf lemma tells us that $|\nabla{w}|\neq 0$ near $\partial\Omega$. Choose $\delta = \delta(\e) > 0$ small enough so that 

\begin{equation}\label{eq3}
|\nabla{w}|\geq C > 0 \hspace{0.05in} \mbox{on}\hspace{0.05in} \Omega_{\delta}, \hspace{0.1in} d_H(\Omega_{\delta}, \p \Omega) < \e.
\end{equation}
 Such a choice is possible because of the non vanishing nature of the gradient which implies continuity of the level sets for small $\delta.$ Combining this with the second estimate from \eqref{eq1}, we obtain that $|\nabla{\tau}|\neq 0$ on $\Omega_\delta$. In particular $\{\tau=0\}$ is a Jordan curve that bounds a simply connected RL-domain, say $D$. By further reducing $\delta$ if necessary, we obtain from \eqref{eq3} that $d_{H}(D, \Omega)< \e. $

\end{proof}

\section{Quadrature domains}\label{sec:quad}

Recall that a bounded domain $\Omega \subset \C$ is called a \emph{quadrature domain} if it admits a formula expressing the area integral of any integrable analytic function $g$
as a finite sum of weighted point evaluations of $g$ and its derivatives:
\begin{equation}\label{eq:QF1}
 \int_{D} {g(z) dA(z)} = \sum_{m=1}^{L} \sum_{k=0}^{n_m}{a_{m,k}g^{(k)}(z_m)},
\end{equation}
where the points  $z_m \in \Omega$ and constants $a_{m,k}$ are each independent of $g$.
The \emph{order} of the quadrature domain $\Omega$ is defined as
the sum $\sum_{m=1}^L n_m$.

Quadrature domains have been studied extensively
in function theory, potential theory, and fluid dynamics,
see \cite{GuSh} for a survey.

We conjecture that the size of the critical set of $v$ is bounded by a function that is linear in the order $n$ of $\Omega$ and the connectivity $k$ of $\Omega$.

\begin{conj}\label{conj:quad}
Let $\Omega$ be a quadrature domain of order $n$
and connectivity $k$.
The number $N$ of critical points
of $v$ satisfies
$ N \leq 2n-2 + k$.
\end{conj}

The topology of quadrature domains was studied in \cite{LeeMak}
where it was shown that
$k \leq 2n-4$ for $n \geq 3$.
Thus, if the above conjecture is true, then we have $N \leq 4n-6$, for $n \geq 3$.

Let us describe a heuristic for constructing examples that attain the conjectured bound (intuition for the conjecture is based on our belief that this construction is extremal).
We recall an interpretation of quadrature domains in terms of logarithmic potential theory.
If $\Omega$ is a quadrature domain
(and suppose the quadrature formula holds for harmonic test functions as well as analytic ones)
then the logarithmic potential generated by area measure over $\Omega$ coincides (outside of $\Omega$) with that which is generated by a finite system of point charges (possibly including dipoles and higher order multi-poles)
situated at the quadrature nodes $z_1,z_2,...,z_L$.
With this formulation, the order of the quadrature domain is the number of poles (counting multi-poles with multiplicity).
Consider a collection of $n$ disjoint disks with each tangent to at least one other disk.
The logarithmic potential generated by the area measure on these disks is equivalent to that of a system of point charges.
Slightly increasing each point charge by $\e>0$, the corresponding quadrature domain (that generates the same logarithmic potential as the system of increased charges) contains the closure of the original set of disks, and for $\e>0$ small we expect the landscape function $v$ to have $n$ maxima.
The number $S$ of saddles is then given by $S = n+k-2$ so that the total number of critical points attains the conjectured bound of $2n+k-2$.
Below, we will perform a detailed analysis of the simplest nontrivial instance of this construction, a quadrature domain whose logarithmic potential is equivalent to two equal charges,
see Section \ref{sec:examples}.

The following result provides an upper bound for $N$ (yet with quadratic growth in $n$) in the special case when $\Omega$ is simply-connected and there is a single node appearing in the quadrature formula,
i.e., we have $L=1$ in \eqref{eq:QF1}.
In this case it is known \cite{AhSh} that the conformal mapping is a polynomial of degree $n$,
and by making a translation we may assume that the point evaluation in \eqref{eq:QF1} is at $\phi(0) = 0$.

\begin{thm}\label{thm:quad}
Let $\Omega$ be a smooth, simply-connected quadrature domain of order $n$ with $L=1$, i.e., the conformal mapping
$\phi:\D \rightarrow \Omega$ is a polynomial of degree $n$.
Then the number $N$ of critical points of $v$ in $\Omega$ satisfies
$N \leq (2n-1)^2.$
\end{thm}

The proof will require the following alternate form of Bezout's theorem.

\begin{thm}\label{thm:Bezout}
Let $A$ and $B$ be polynomials
in the real variables $x$ and $y$ with real coefficients with $\deg A$ = $n_1$ and $\deg B$ = $n_2$.
If $A$ and $B$ have finitely many common zeros in a domain $\Omega$ then they have at most $n_1 \cdot n_2$ common zeros in $\Omega$.
\end{thm}

\begin{proof}[Proof of Theorem \ref{thm:Bezout}]
The result follows from the exact same argument as given in the proof of \cite[Thm. 2]{Wilm}, but instead of considering all real zeros of a hypothetical common factor $C(x,y)$, one considers only the real zeros of $C(x,y)$ within $\Omega$.
\end{proof}


\begin{proof}[Proof of Theorem \ref{thm:quad}]
Let $u$ denote the solution to \eqref{eq:harm},
and recall that the critical points
of $v$ are solutions to $\nabla u = z$.
We use the conformal mapping $\phi$ to pull back $u$ to a harmonic function $\hat{u}(w) = u(\phi(w))$
that solves
\begin{equation}
 \begin{cases}
\Delta \hat{u}(w) = 0, & w\in\D \\
\hat{u}(w)= \frac{|\phi(w)|^2}{2}, & w\in\partial\D.
\end{cases}
\end{equation}
As shown in \cite[proof of Thm. 1.1]{FlLu}, 
the solution to this Dirichlet problem 
satisfies $\hat{u}(w) = P(w) + \ol{P(w)}$,
where $P$ is a polynomial of degree $n-1$.

The solutions of $\nabla v = 0$,
or $\nabla u(z) = z$ correspond, in $w$ coordinates, to solutions of
$$ (\nabla u ) (\phi(w)) = \phi(w).$$
Since $P'(w) = \nabla \hat{u} (w) = (\nabla u ) (\phi(w)) \cdot \ol{\phi'(w)}$, we rewrite the above as
\begin{equation}
     \frac{\ol{P'(w)}}{\ol{\phi'(w)}} = \phi(w) ,
\end{equation}
or, upon clearing the denominator,
\begin{equation}\label{eq:system}
     \ol{P'(w)} - \phi(w)\ol{\phi'(w)} = 0.
\end{equation}
The critical points of $v$ in $\Omega$ correspond to the zeros of \eqref{eq:system} in the unit disk $\D$.
Since $\Omega$ is simply-connected, 
Theorem \ref{thm:finite} implies that the number of such zeros is finite.
Taking real and imaginary parts in \eqref{eq:system}, 
we arrive at a system of two real polynomial equations in two variables with each polynomial being of degree $2n-1$.
Applying Theorem \ref{thm:Bezout}
we see that there are at most $(2n-1)^2$ many solutions of \eqref{eq:system} in $\D$ and hence $v$ has at most this many critical points in $\Omega$.
\end{proof}

\section{Additional Examples}\label{sec:examples}

In the following example,
we determine the
number of critical points of $v$
for a one-parameter family of simply-connected quadrature
domains of order $n=2$.

\subsection{Neumann's oval}
Let $\Omega\subset\mathbb{C}$ be the domain whose boundary is given by 
$$\{(x,y): (x^2 +y^2)^2 = a^2(x^2 + y^2) + 4x^2\}$$
for some parameter $a>0.$ Let $R = \frac{a + \sqrt{a^2 + 4}}{2}.$ Note that $R>1.$ It is known that $\phi(z) = \frac{1-R^4 + \sqrt{(R^4-1)^2 + 4R^4z^2}}{2Rz}$ maps $\Omega$ conformally onto the unit disk $\mathbb{D}$ and that the landscape function of $\Omega$ is given by \cite{FlSi}

\begin{equation}
v(z) = \dfrac{R^4 - 1}{2R^2} + \dfrac{\Re(z\phi(z))}{R} - \dfrac{|z|^2}{2}.
\end{equation}

\noindent A routine computation gives
\begin{equation}\label{grad}
\nabla v(z) = \overline{\dfrac{2R^2z}{\sqrt{(R^4-1)^2 + 4R^4z^2}}} - z.
\end{equation}
We claim that there exists $R_0 > 1$ such that if $1< R\leq R_0,$   $\nabla v$ has exactly 3 zeros (counting multiplicity), all of which are real, and if $R> R_0$ there is only one zero, at the origin.

 For this, we let $z=t\in\mathbb{R}$  in equation \eqref{grad}. Then we obtain $t\left(2R^2- \sqrt{(R^4-1)^2 + 4R^4t^2}\right) = 0 $
whose solutions are $t=0,$ and the roots of
$$t^2 = \dfrac{4R^4 - (R^4-1)^2}{4R^4}.$$
It now follows that if $4R^4 - (R^4-1)^2\geq 0$, i.e., if $1< R\leq R_0=\sqrt{1+\sqrt{2}},$ equation \eqref{grad} has three real zeros. Namely, one positive root, one negative root, and a root at $0$ if $1< R< R_0, $ and a root of multiplicity $3$ at the origin $t=0,$ if $R= R_0.$ If $R > R_0$ it is clear that the only real root is at $0.$ It remains to be shown that equation \eqref{grad} has no nonreal zeros.

Suppose that $z\notin\mathbb{R}$ is a critical point of $v,$ i.e, 

\begin{equation}\label{g}
 2R^2z = \bar{z}\sqrt{(R^4-1)^2 + 4R^4z^2}.
\end{equation}
We first show that $z$ cannot be purely imaginary. Indeed, substituting $z=it$, $t\neq 0$, into equation \eqref{g} yields $2R^2 = -\sqrt{(R^4-1)^2 - 4R^4t^2}$ which is clearly impossible, since $t$ is supposed to be a real number. Now assume that $z=te^{i\theta}$ is a critical point of $v$, where $\theta\in (0, 2\pi),$ with $\theta\notin\{\frac{\pi}{2}, \pi\}$. Substituting in equation \eqref{g} and squaring, we obtain 

\begin{equation}\label{sq}
4R^4 = e^{-i4\theta}(R^4-1)^2 + 4R^4t^2e^{i2\theta}.
\end{equation}

Equating the imaginary parts in equation \eqref{sq}, we obtain $(R^4-1)^2\sin(4\theta)= -4R^4t^2\sin(2\theta)$. Writing $\sin(4\theta)$ as $2\sin(2\theta)\cos(2\theta)$ and remembering that $\theta\notin\{0, \frac{\pi}{2}, \pi\}$, we obtain

\begin{equation}\label{im}
\cos(2\theta) = -\dfrac{2t^2R^4}{(R^4-1)^2}.
\end{equation}

Similarly, equating the real parts of equation \eqref{sq}, and solving for $\cos(2\theta)$ yields,

\begin{equation}\label{re}
\cos(2\theta) = \dfrac{1}{4(R^4-1)^2}\left( -4R^4t^2 \pm \sqrt{(4R^4t^2)^2 + 8(R^4-1)^2\left[(R^4-1)^2 + 4R^4 \right]}\right).
\end{equation}

Clearly equations \eqref{im} and \eqref{re} are incompatible. This shows that all critical points of $v$ are real and finishes the proof of our claim.

\qed

\begin{remark}
The Neumann's oval is an example of a Hippopede. From the classical theory of plane curves it is known that Hippopedes can be parametrized in polar coordinates \cite[p. 145]{La}, showing that the Neumann's oval is star-shaped with respect to the origin. Our example above therefore furnishes a star-shaped domain in which there are two maxima (near the $2$ nodes) for the landscape function.  In particular, this answers a question in the survey \cite[p. 19]{Magnanini}, where the author asks if there is a unique maximum for the landscape function of a star-shaped domain. As we mentioned before, in the case of convex domains the existence of a unique critical point is well known. 
\end{remark}

\begin{lemma}
For $\e > 0$ denote $N_{\e} = (-1, 1)\times(-\e, \e),$ and let $M\gg 1.$ Let $w: N_{\e}\rightarrow\mathbb{R}$ satisfy 

\[ \begin{cases}
\Delta w(z) = -2, & z\in N_{\e} \\
w(z) =0, & z\in\hspace{0.05in}\mbox{top and bottom sides},\\
w(z) = M, & z\in\hspace{0.05in}\mbox{left and right sides}.
\end{cases}
\]
Then $$\sup_{y\in [-\e, \e]}w(0, y)\to 0\hspace{0.05in}\mbox{as}\hspace{0.05in}\e\to 0.$$
\end{lemma}

We will refer to $N_{\e}$ as an $\e-$neck.

\begin{proof}
Let $w_1(z) = w(z) + y^2 - \e^2.$ Then $\Delta w_1(z) = 0$ in $N_{\e},$ with boundary values
\[ \begin{cases}
w_1(z) =0, & z\in\hspace{0.05in}\mbox{top and bottom sides},\\
w_1(z)\leq M, & z\in\hspace{0.05in}\mbox{left and right sides}.
\end{cases}
\]
If $h$ is defined on $N_{2\e}$ by 
$$h(x, y) = C_{\e}\cosh\left(\frac{\pi x}{4\e}\right)\cos\left(\frac{\pi y}{4\e}\right)$$
where $C_{\e} =\frac{2M}{\cosh(\frac{\pi}{4\e})},$ then $h$ is positive harmonic in $N_{2\e}$. Furthermore, it is easy to see that on $\partial N_{\e}$, $w_1\leq h.$ By the maximum principle, this forces $w_1\leq h$ in $N_{\e}$. This in turn implies,
$$ w(z)\leq h(z) + \e^2 - y^2\leq h(z) + \e^2,\hspace{0.05in} z\in N_{\e}.$$ In particular, taking $z = (0, y)\in N_{\e}$, we obtain 
$$ w(0, y)\leq C_{\e}+ \e^2\to 0\hspace{0.05in}\mbox{as}\hspace{0.05in}\e\to 0$$
\end{proof}

As an application, consider $\Omega,$ a dumbbell shaped domain consisting of two almost discs, attached by an $\e-$neck for a very small $\e >0.$ Let $v$ be the landscape function on $\Omega.$ Then taking $M$ to be a very large number and applying the Lemma, we obtain $v$ is very small in the middle of this neck. Hence $v$ has at least one maximum in each disc, for a total of at least $3$ critical points ($2$ maxima and a saddle) in $\Omega$.

\noindent This construction can be generalized to a region comprised of $n$ almost disks held together by $n-1$ $\e-$necks. This will yield that the landscape function on such a domain has at least $2n -1$ critical points. 

\section{Concluding Remarks}
\label{sec:concl}

\subsection*{Critical points of $v$ and oscillations of $\partial \Omega$}

Let us restrict to the case $\Omega$ is simply-connected,
and also assume that $\partial \Omega$ is smooth and real-analytic.
Then the conformal mapping $\phi: \D \rightarrow \Omega$ from the unit disk extends to be analytic in $\ol{\D}$.
Let $n$ denote the number of
local maxima of $|\phi|$ over $\partial \D$. 
Geometrically, $n$ counts oscillations of the distance from the boundary $\partial \Omega$ to the origin.
Thus an affirmative answer to the following problem would provide an estimate based on a purely geometric feature of $\Omega$.

\begin{prob}
Let $\Omega$ be a simply-connected domain with a conformal mapping $\phi : \D \rightarrow \Omega$ having $n$ boundary oscillations as explained above.
Prove an upper bound for the number $N$ of critical points of $v$ in terms of $n$.
\end{prob}

We conjecture that $N$ increases at most linearly with $n$, and one specific guess for a bound of this form (without yet sufficient evidence to state as a conjecture) is 
\begin{equation}\label{eq:guess}
 N \leq 2n-1.
\end{equation}

Let us explain the intuition behind the guess \eqref{eq:guess}.
Let $u$ be the harmonic function
that solves \eqref{eq:harm},
and let $F(z)$ denote the complex conjugate of $\nabla u$.
The assumptions on $\p \Omega$
imply that $F(z)$ extends analytically to a neighborhood of $\ol{\Omega}$.
The number of critical points of $F$
in $\Omega$ is at most $n-1$.
Indeed, using the conformal mapping 
$\phi:\D \rightarrow \Omega$
we pull back $u$ to a harmonic function $\hat{u}$
that solves
\[\label{eq:pullback} \begin{cases}
\Delta \hat{u}(w) = 0, & w\in\D \\
\hat{u}(w)= \frac{|\phi(w)|^2}{2}, & w\in\partial\D.
\end{cases}
\]
It is known \cite{AlMa94}
that the number of critical points of $\hat{u}$ in $\D$ is at most $n-1$, where as above $n$ denotes the number of local maxima of the boundary data.
If it could be shown that the number $M$ of maxima of $F$
is bounded by the number $n-1$ of critical points of $F$ as in the anti-holomorphic dynamics schema (as used in the proof of Theorem \ref{thm:RL} above),
then the bound (\ref{eq:guess})
would indeed follow from Theorem \ref{thm:Morse}.

Such an application of anti-holomorphic dynamics could be realized by addressing the following problem in quasiconformal surgery.

\begin{prob}\label{prob:ext1}
Find conditions on $\Omega$ that
guarantee $F(z)$ extends to $\C$
as a quasirational mapping of degree $n$.
\end{prob}

By ``quasirational'' we mean
that the extended mapping $F$
is quasiconformally conjugate to a rational mapping
\cite{BrFa}.
Such an extension allows iteration of $F$ and application of anti-holomorphic dynamics.

The following problem concerns a
weaker extension of $F$ that would yet result in an upper bound for $N$ that is quadratic in $n$ (explained below).

\begin{prob}\label{prob:ext2}
Find conditions that
guarantee $F(z)$
extends to $\C$
as a smooth (but not necessarily complex-analytic) proper orientation-preserving mapping of degree $n$ and asymptotic to
$z^n$ as $z \rightarrow \infty$.
\end{prob}

Suppose such an extension holds. 
Then we can estimate the number of critical points of $v$ as follows.
The critical points of $v$
are solutions of the fixed point problem
$$\ol{F(z)} = z $$
which are also
solutions of the fixed point problem
$$\ol{F(\ol{F(z)})}=z,$$
or equivalently, the zero-finding problem
$$\ol{F(\ol{F(z)})} - z=0. $$
Since $\ol{F(\ol{F(z)})} - z$ is orientation-preserving, we can count the number of zeros in a large disk
using the argument principle.
The winding number is $n^2$,
since $\ol{F(\ol{F(z)})}$
is asymptotic to $z^{n^2}$ on the boundary of a large disk.
Hence, $v$ has at most $n^2$ critical points.

\smallskip

\subsection*{Circular domains}

The following conjecture is
partly inspired by the outcome for RL-domains.

\begin{conj}\label{conj:convex}
Suppose $\Omega$ is a $k$-connected circular domain.
Then for $k \geq 3$
the number $N$ of critical points of $v$ satisfies
$$N \leq 5(k-2).$$
\end{conj}

The term \emph{circular domain} refers to a domain where each boundary component is a circle.

\subsection*{Magnanini's conjecture}

We conclude by recalling the following conjecture stated by R. Magnanini
(see \cite{Magnanini} for discussion).

\begin{conj}\label{conj:Magn}
Suppose $\Omega$ is simply-connected.
The number $N$ of critical points of $v$ in $\Omega$ satisfies
$$ N \leq 2m-1,$$
where $m$ denotes the number of maxima in $\Omega$ of the function $d_{\p \Omega}(z)$ defined as the distance from $z$ to $\p \Omega$.
\end{conj}

To our knowledge, this conjecture is still wide open.
There is no algebraicity condition imposed in the conjecture, yet
it may be worth exploring the conjecture within some interesting classes of domains with algebraic boundary,
since the definition of $m$ in the conjecture is reminiscent of the
so-called \emph{Euclidean distance degree}
studied in real algebraic geometry \cite{Draisma2016}.

One can generalize Conjecture \ref{conj:Magn} to the multiply-connected case
by using $m$ as a conjectural bound on the number of maxima of $v$ and then applying Theorem \ref{thm:Morse}.
For $k$-connected domains, this results in a conjectured bound
of $N \leq 2m-2 + k$
(note the resemblance to Conjecture \ref{conj:quad}).
Relating this multiply-connected version of the conjecture to Conjecture \ref{conj:convex}
raises a basic question in the setting of Computational Geometry:
How many locally maximal disks can be placed within a $k$-connected circular domain?
We were unable to find an answer to this question, but an upper bound of $2(k-2)$
seems plausible
and would
give $2m-2 + k = 5(k-2)$.

\bibliographystyle{abbrv}
\bibliography{crit}

\end{document}